\documentclass[letterpaper, conference]{IEEEtran}
\usepackage{fullpage,graphicx,url,setspace}

\usepackage{xparse}
\usepackage{cite}
\usepackage{bm}
\usepackage{color}
\usepackage[small,bf]{caption}
\usepackage{amsmath,verbatim,amssymb,amsfonts,amscd, graphicx, algorithmic, algorithm}
\usepackage{amsmath,amsthm}
\usepackage{mathtools}
\mathtoolsset{mathic}
\usepackage{microtype}
\usepackage{wrapfig}

\usepackage{ifxetex,ifluatex}
\ifxetex
\usepackage{xltxtra}
\else
\ifluatex
\usepackage{luatextra}
\else
\usepackage{amssymb}
\usepackage{algorithmic, algorithm}
\usepackage[utf8]{inputenc}

\csname else\expandafter\endcsname
\romannumeral -`0
\fi
\fi
\iftrue\relax%
\defaultfontfeatures{Ligatures=TeX}
\usepackage[math-style=TeX, vargreek-shape=TeX]{unicode-math}
\fi

\NewDocumentCommand{\entropy}{om}{\mathbb{H}\left[#2
    \IfValueT{#1}{\,\middle|\,#1}\right]}
\NewDocumentCommand{\bentropy}{lm}
  {\widetilde{\mathbb{H}}#1\left[#2\right]}

\NewDocumentCommand{\mutualInfo}{omm}{\mathbb{I}\left[#2;#3
    \IfValueT{#1}{\,\middle|\,#1}\right]}

\usepackage{tikz}
\usetikzlibrary{shapes.geometric, positioning, shadows}
\usepackage{graphicx,color}

\newtheorem{theorem}{Theorem}

\newtheorem{lemma}{Lemma}

\newtheorem{remark}{Remark}

\usepackage{enumitem}
\newlist{enumerate*}{enumerate*}{1}
\setlist[enumerate*]{label=(\arabic*)}

\newcommand{\ben}{\begin{eqnarray}}

\newcommand{\een}{\end{eqnarray}}

\title{Categorical Matrix Completion}

\author{\IEEEauthorblockN{Yang Cao and Yao Xie}
\IEEEauthorblockA{H. Milton Stewart School of Industrial and Systems Engineering, Georgia Institute of Technology\\
\{caoyang, yao.xie\}@gatech.edu}
}

\begin{document}
\maketitle

\begin{abstract}

We consider the problem of completing a matrix with categorical-valued entries from partial observations. This is achieved by extending the formulation and theory of one-bit matrix completion  \cite{davenport20141}. We recover a low-rank matrix $X$ by maximizing the likelihood ratio with a constraint on the nuclear norm of $X$, and the observations are mapped from entries of $X$ through multiple link functions. We establish theoretical upper and lower bounds on the recovery error, which meet up to a constant factor $\mathcal{O}(K^{3/2})$ where $K$ is the fixed number of categories. The upper bound in our case depends on the number of categories implicitly through a maximization of terms that involve the smoothness of the link functions. In contrast to one-bit matrix completion, our bounds for categorical matrix completion are optimal up to a factor on the order of the square root of the number of categories, which is consistent with an intuition that the problem becomes harder when the number of categories increases. By comparing the performance of our method with the conventional matrix completion method on the MovieLens dataset, we demonstrate the  advantage of our method.

\end{abstract}

\section{Introduction}

Recovering a low-rank matrix $M$ from a subset of its entries is a fundamental problem that arises from many real-world applications. The so-called matrix completion problem was originally formulated as estimating a matrix $M$ with real-valued entries, subjecting to the data fit constraint \cite{candes2009exact,candes2010matrix}. However, in many problems entries are categorical (e.g., in recommender systems the ratings take integer values 1 to 5, or in health care applications, where the results are positive, negative or uncertain.) A better formulation for these scenarios would be {\it categorical matrix completion.}

In this paper, we consider categorical matrix completion by extending the formulation of  one-bit matrix completion \cite{davenport20141} to deal with categorical entries and adopt the proof techniques to obtain upper and lower bounds. Assume the input variables form a low-rank matrix $X$, and we observe partial entries of a matrix which are categorical responses of the underlying low-rank matrix. A new problem arises in the categorical setting is to choose appropriate link functions $f_k$'s that map entries of $X$ to entries of the observed matrix $M$. We consider multinomial logistic regression link functions, which are smooth 
and they are easy to construct for an arbitrary number of categories.
We consider a nuclear norm regularized maximum likelihood estimator with a likelihood function for categorical distribution (different from the Bernoulli distribution used in the one-bit case). To obtain theoretical upper and lower bounds, we introduce new conditions taking in account of the characteristics of the categorical distribution.
Our upper and lower bounds match up to a factor that is on the order of the square root of the number of categories. 
Finally, we compare the performance of our method with the convention matrix completion method on the MovieLens dataset.

As mentioned, a closely related work is one-bit matrix completion \cite{davenport20141}, where the matrix entries are binary valued 
and therein the authors establish theoretical upper and lower bounds for the mean squared error of the recovered matrix which demonstrates the optimality of the estimator.
Recently, \cite{klopp2014adaptive} considers matrix completion over finite alphabet with a nuclear norm regularization, and considers a more general sampling model that only requires knowledge about an upper bound for the entries of the matrix; a theoretical upper bound is given therein, which has a faster convergence rate than that in \cite{davenport20141}; there is no theoretical lower bound though. A more recent work \cite{lafond2014probabilistic} provides a lower bound for the special case when $K = 2$. Other related work on matrix completion with quantized entries or Poisson observations include \cite{cai2013max,soni2014noisy,soniestimation,cao2015poisson,lafond2015low}.

The rest of this paper is organized as follows. Section II sets up the formalism for categorical matrix completion and the nuclear norm regularized maximum likelihood estimator. Section III establishes the upper and lower bounds for the recovery error. Section IV presents an numerical example using the MovieLens dataset to demonstrate the performance of our method. All proofs are delegated to the appendix\footnote{Full version of the paper can be downloaded from \\{www2.isye.gatech.edu/$\sim$yxie77/Categorical-MC-CAMSAP.pdf}}.

The notation in this paper is standard. In particular, $[d] =\{1,2,\ldots,d\}$; $\mathbb{I}_{[\varepsilon]}$ is the indicator function for an event $\varepsilon$; $|A|$ denotes the number of elements in a set $A$. Let entries of a matrix $M$ be denoted by $M_{ij}$ or $[M]_{ij}$, $\|M\|$ be the spectral norm which is the largest absolute singular value, $\|M\|_{F} = \sqrt{\sum_{i,j}M_{ij}^2}$ be the Frobenius norm, $\|M\|_*$ be the nuclear norm which is the sum of the singular values and finally $\|M\|_{\infty}$ = $\max_{ij}|M_{ij}|$ be the infinity norm. Let $\mbox{rank}(M)$ denote the rank of a matrix $M$. The inner product for two matrices $M_1$ and $M_2$ is denoted by $\langle M_1, M_2 \rangle \triangleq \mbox{tr}(M_1^T M_2)$. Given the number $K$ of categories and any set $\{a_1,a_2,\ldots,a_K\}$, we say that a random variable $X$ satisfies the categorical distribution with the parameters $(p_1,p_2,\ldots,p_K)$ if $\sum_{k=1}^K p_k=1$ and $\mathbb{P}(X=a_k) = p_k$ for all $k\in [K]$. Also define the Kullback-Leibler (KL) divergence between two categorical distributions with parameters $(p_1,\ldots,p_K)$ and $(q_1,\ldots,q_K)$ as
\[
D\left( (p_1,p_2,\ldots,p_K) \| (q_1,q_2,\ldots,q_K) \right) \triangleq \sum_{k=1}^K p_k\log\frac{p_k}{q_k},
\]
and define their Hellinger distance as
\[
d_H^2\left( (p_1,p_2,\ldots,p_K) , (q_1,q_2,\ldots,q_K) \right) \triangleq \sum_{k=1}^K \left(\sqrt{p_k}-\sqrt{q_k}\right)^2.
\]

\section{Formulation}

Suppose we make noisy observations of a matrix $M \in \mathbb{R}^{d_1 \times d_2}$ on an index set $\Omega \subset [d_1] \times [d_2]$. The indices are randomly selected with $\mathbb{E}|\Omega|=m$, or, equivalently, the indicator functions $\mathbb{I}_{\{(i,j) \in \Omega\}}$ are i.i.d. Bernoulli random variables with parameter  $m/(d_1 d_2)$. Assume that the observed entries take one of the $K$ possible values: $\{a_1,a_2,\ldots,a_K\}$. Given a set of differentiable link functions $f_k$, $k = 1, \ldots, K$ that satisfy $\sum_{k=1}^K f_k(x) = 1$, we have that the noisy observations follow the categorical distribution:
\begin{equation}
Y_{ij} = a_k \mbox{ with probability } f_k(M_{ij}) \mbox{ for } (i,j) \in \Omega.
\label{observationmodel}
\end{equation}
Our goal is to recover $M$ from the categorical observations $\{Y_{ij}\}_{{(i,j)}\in \Omega}$ and further filling the missing entries using the link functions and the entries of recovered matrix $M$. This is done by letting $Y_{ij}=a_{k^*}$,  for all $(i,j) \in ([d_1]\times [d_2])/\Omega$, where $k^* = \arg \max_{1\leq k \leq K} f_k(M_{ij})$.

The following are two simple illustrative examples for link functions. In a $K$-categorical recommender system, there are $K$ possible ratings, and the matrix with entries $M_{ij} \in [K]$ is the true rating matrix of $d_1$ users for $d_2$ items.   Suppose users can be in three possible moods: good, normal, and bad. The link function characterizes the bias of a user and we can observe a subset of biased ratings. Suppose a user tends to rate an item one category lower than the truth in a bad mood, and one category higher than the truth in a good mood, with the probabilities of being in bad, normal, and good mood being 0.2, 0.6, and 0.2, then the link functions are given by
\begin{equation*}
\left\{
\begin{array}{l}
f_1(1) = 0.8; f_1(2) = 0.2; f_1(x) = 0, \mbox{otherwise.}\\
f_k(k-1) = 0.2; f_k(k) = 0.6; f_k(k+1) = 0.2; \\
f_k(x) = 0, \mbox{otherwise, }\quad k = 2, \ldots, K-1;\\
f_K(K-1) = 0.2; f_K(K) = 0.8; f_1(x) = 0, \mbox{otherwise.}
\end{array}\right.
\end{equation*}
The second example is the widely used proportional-odds cumulative logit model, or multinomial logistic model \cite{agresti2011categorical}, where
\begin{equation}
f_k(x) \propto e^{\alpha_k+\beta_k x}, \quad k=[K],  
\label{multi-nomial}
\end{equation}
and $\sum_{k=1}^K f_k(x) = 1$. Here $\alpha_k$ and $\beta_k$ are parameters of the model that are given (or obtained from a training stage).

In addition, we make assumptions for the matrix $M$ to be recovered. First, we assume an upper bound $\alpha$ for $\|M\|_{\infty}$ to entail the recovery problem is well posed. 
Second, similarly to the conventional matrix completion, we assume that the nuclear norm of the matrix is bounded $\|M\|_* \leq \alpha \sqrt{rd_1 d_2}$. This assumption can be viewed as a relaxation of $\|M\|_{\infty} \leq \alpha$ and rank$(M) \leq r$ \cite{davenport20141}, since $\|M\|_* \leq \sqrt{\mbox{rank}(M)}\|M\|_F$ and $\|M\|_F \leq \sqrt{d_1 d_2}\|M\|_{\infty}$ lead to $\|M\|_* \leq \alpha \sqrt{r d_1 d_2}$.

To estimate $M$, we consider the following nuclear norm regularized maximum log-likelihood formulation. In our case, the log-likelihood function is given by
\[
F_{\Omega,Y}(X) \triangleq \sum_{(i,j)\in \Omega} \sum_{k=1}^K \mathbb{I}_{[Y_{ij}=a_k]}\log(f_k(X_{ij})).
\]
Based on the assumptions above, we consider a set $\mathcal{S}$ of candidate estimators:
\begin{equation}
\begin{split}
&\mathcal{S} \triangleq \left\{ X \in \mathbb{R}_+^{d_1 \times d_2} : \|X\|_* \leq \alpha \sqrt{r d_1 d_2}, \right. \\
&\qquad \qquad \qquad \qquad \left. -\alpha \leq X_{ij} \leq \alpha, \forall (i,j) \in [d_1] \times [d_2] \right\},
\end{split}
\label{searchspace}
\end{equation}
and recover $M$ by solving the following optimization problem
\begin{equation}
\begin{split}
&\widehat{M} = \arg \max_{X \in \mathcal{S}} F_{\Omega,Y}(X).
\end{split}
\label{optimization_problem}
\end{equation}
This problem is convex and it can be solved exactly by the interior-point method \cite{liu2009interior} or approximately by the efficient singular value thresholding method \cite{cai2010singular}.

\label{sec:formulation}
\section{Performance bounds}

To establish our performance bounds, we make the following assumptions on the link functions  $f_k$. 
Define for any $k \in [K]$ and a region $x\in [-\alpha, \alpha]$
\begin{equation*}
L_{\alpha}^{(k)} \triangleq \sup_{|x|\leq \alpha} \frac{|f_k'(x)|}{f_k(x)},
\quad
\beta_{\alpha}(x) \triangleq \max_{1\leq k \leq K} \frac{(f_k'(x))^2}{f_k(x)}.
\end{equation*}
Assume that (1) there exists a positive constant $L_{\alpha}$ such that
\begin{equation}
\max_{1\leq k\leq K} L_{\alpha}^{(k)} \leq L_{\alpha}.
\label{constant1}
\end{equation}
The interpretation of this assumption is that the function $f_k(x)$ does not change sharply when it is near the boundaries of the region;
and (2) there exist two positive constants $\beta_{\alpha}^{-}$ and $\beta_{\alpha}^{+}$ such that
\begin{equation}
\inf_{|x|\leq \alpha} \beta_{\alpha}(x) \geq \beta_{\alpha}^{-} \quad \mbox{and} \quad \sup_{|x|\leq \alpha} \beta_{\alpha}(x) \leq \beta_{\alpha}^{+}.
\label{constant2}
\end{equation}
This lower bound for $\beta_{\alpha}(x)$ means that for every fixed $x \in [-\alpha, \alpha]$, there exists at least one $k \in [K]$ such that $f_k$ does not change too slowly. Another interpretation for the assumption on the lower bound is that $f_k$'s overlap moderately so that we may determine the category uniquely for a given $x\in [-\alpha,\alpha]$. The interpretation of upper bound for $\beta_{\alpha}(x)$ is similar to that for the upper bound $L_{\alpha}^{(k)}$. When $K = 2$, these assumptions coincide with those in \cite{davenport20141}. 

Many link functions satisfy the previous two assumptions, including the widely used multinomial logistic model (\ref{multi-nomial}). 
Fig. \ref{fig:function} illustrates one such example of link functions where $\alpha=10$ and $K=5$.
\begin{figure}[h]
\begin{center}
\includegraphics[width = 0.8\linewidth]{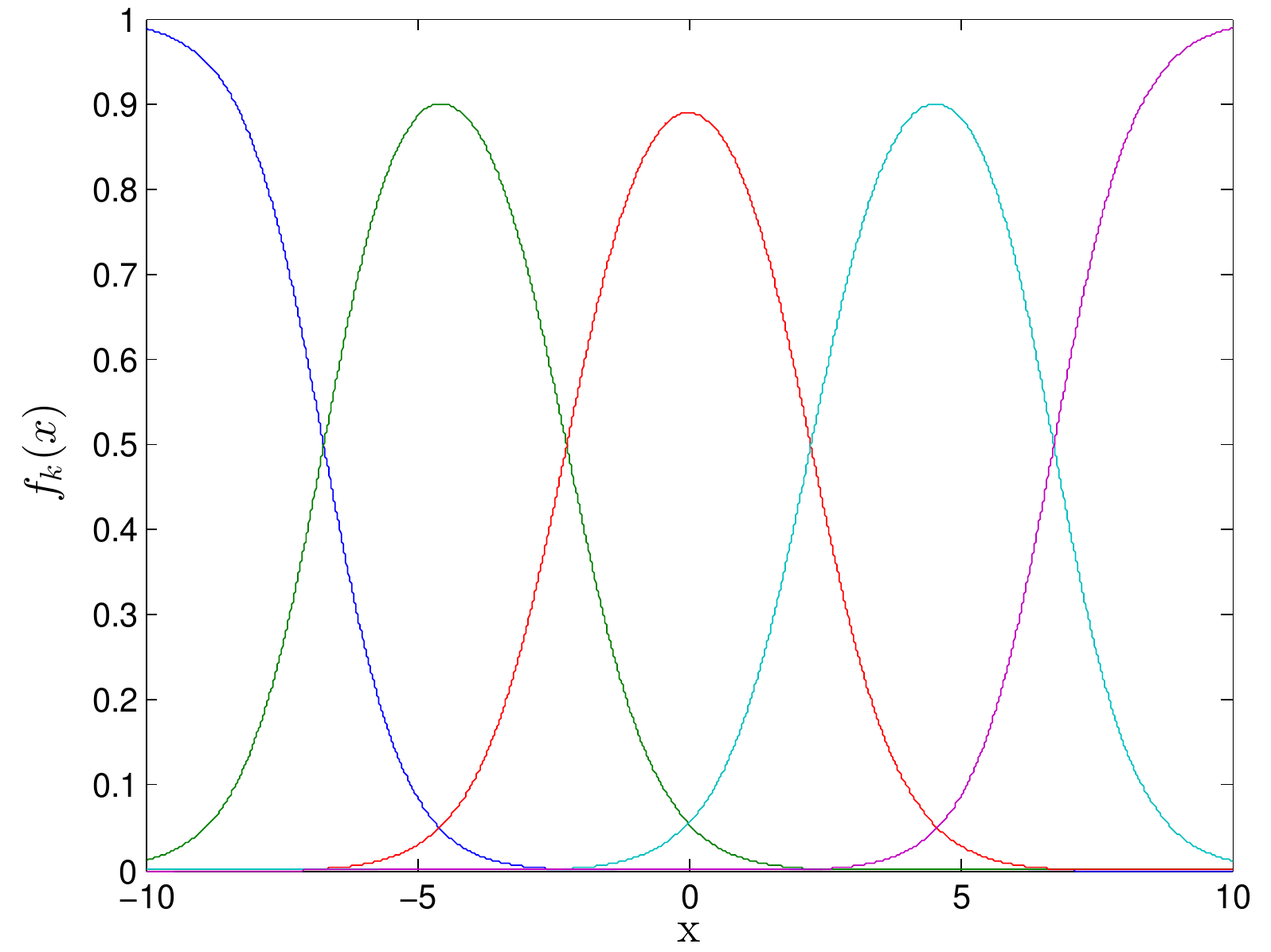}
\caption{$f_k(x), ~k \in [K]$ with $\alpha=10$ and $K=5$.}
\label{fig:function}
\end{center}
\vspace{-0.2in}
\end{figure}
Furthermore, define the average Hellinger distance and KL divergence for entries of two matrices $P,Q \in \mathbb{R}^{d_1 \times d_2}$ as:
\begin{equation}
\begin{split}
d_H^2\left(f(P) , f(Q) \right)
=&\frac{1}{d_1 d_2} \sum_{i,j} d_H^2 \left( f(P_{ij}), f(Q_{ij}) \right),\\
D\left(f(P) \| f(Q) \right)
=&\frac{1}{d_1 d_2} \sum_{i,j} D \left( f(P_{ij})\| f(Q_{ij}) \right).
\end{split}
\end{equation}

The following two lemmas are needed to prove the upper bound. To use the contraction principle in Lemma \ref{firstlemma}, we introduce a function
\begin{equation}
\bar{F}_{\Omega,Y}(X) = F_{\Omega,Y}(X) - F_{\Omega, Y}(0).
\label{likelihood}
\end{equation}
\begin{lemma}
Let $\bar{F}_{\Omega, Y}(X)$ be the likelihood function defined in (\ref{likelihood}) and $\mathcal{S}$ be the set defined in (\ref{searchspace}), then
\begin{equation}
\begin{split}
&\mathbb{P} \left\{ \sup_{X\in \mathcal{S}} \left| \bar{F}_{\Omega, Y}(X)-\mathbb{E}\bar{F}_{\Omega, Y}(X)\right| \right.\\
& \quad \left. \geq C'K L_{\alpha} \alpha\sqrt{r} \cdot \left(\sqrt{m(d_1+d_2)+d_1 d_2 \log(d_1 d_2)}\right) \right\} \\
& \leq \frac{C}{d_1 d_2},
\end{split}
\end{equation}
where $C'$ and $C$ are absolute positive constants and the probability and the expectation are both over $\Omega$ and $Y$.
\label{firstlemma}
\end{lemma}


\begin{lemma}
For $M_{ij}$ and  $\widehat{M}_{ij}$ both in $[-\alpha, \alpha]$, $\forall (i,j) \in [d_1] \times [d_2]$, we have
\[
d_H^2(f(M),f(\widehat{M})) \geq \frac{\beta_{\alpha}^-}{4} \frac{\|M-\widehat{M}\|_F^2}{d_1 d_2}.
\]
\label{secondlemma}
\end{lemma}
Our main results are the upper bound for the average mean square error per-entry in Theorem \ref{maintheorem}, and an information theoretic lower bound in Theorem \ref{maintheorem2}:
\begin{theorem}[Upper bound]
\label{maintheorem}
    Assume $M \in \mathcal{S}$, and $\Omega$ is chosen at random following the binomial sampling model with $\mathbb{E}[|\Omega|] = m$. Suppose that $Y$ is generated as in (\ref{observationmodel}). Let $L_{\alpha}$ and $\beta_{\alpha}^-$ be as in (\ref{constant1}) and (\ref{constant2}). Let $\widehat{M}$ be the solution to (\ref{optimization_problem}). Then with a probability exceeding $\left(1-C/(d_1 d_2)\right)$, we have
    \begin{equation}
    \begin{split}
       & \frac{1}{d_1 d_2} \|M-\widehat{M}\|_F^2 \leq \frac{C' \alpha K L_{\alpha}}{\beta_\alpha^-}
       \cdot \\
       & \qquad \sqrt{\frac{r(d_1 +d_2)}{m}} \sqrt{1+\frac{(d_1+d_2)\log(d_1 d_2)}{m}}.
    \end{split}
    \label{bound:MC}
    \end{equation}
    If $m\geq (d_1+d_2)\log(d_1 d_2)$ then (\ref{bound:MC}) simplifies to
    \begin{equation}
    \begin{split}
       & \frac{1}{d_1 d_2} \|M-\widehat{M}\|_F^2 \leq \frac{\sqrt{2}C' \alpha K L_{\alpha}}{\beta_\alpha^-}  \sqrt{\frac{r(d_1 +d_2)}{m}}.
    \end{split}
    \label{bound:MC2}
    \end{equation}
Above, $C, C'$ are absolute constants.
\end{theorem}

\begin{remark}
The ratio $L_{\alpha}/\beta_{\alpha}^-$ depends on the number of categories $K$ implicitly though the maximization of the smoothness of the functions $f_k$. 
\end{remark}

\begin{remark}
For a fixed $\alpha$, we can construct function $f_k$'s such that the ratio $KL_{\alpha}/\beta_{\alpha}^-$ is less than some absolute constant for any given $K$. In other words, we may be able to choose the link functions $f_k$ such that the upper bound is independent of the number of categories. Therefore, how to choose $f_k$ that satisfies the classification requirement as well as minimizing this ratio becomes important. Examining the first inequality in (\ref{se1}), a good choice for $f_k$ should be that for any $x,y \in [-\alpha, \alpha]$, $x\neq y$,  there exists at most one $k \in [K]$ such that $f_k(x) = f_k(y)$. Fortunately, such $f_k$ is not hard to construct and one such example is the multinomial logistic model, as demonstrated in Fig. \ref{fig:function}.
\end{remark}

\begin{remark}
Given $K$,$f_k$ and $\alpha$, the mean squared error per entry in (\ref{bound:MC2}) tends to $0$ with probability $1$ as the dimensions of the matrix $M$ goes to infinity and $r=o(\log(d_1 d_2))$. In other words, one can recover accurately with a sufficiently large number of observations.
\end{remark}
%
%
The following lemmas are used in proving the lower bound.
\begin{lemma}[Lemma A.3 in \cite{davenport20141}]
Let $\mathcal{S}$ be as in (\ref{searchspace}). Let $\gamma \leq 1$ be such that $r/\gamma^2$ is an integer.
Suppose  $r/\gamma^2 \leq d_1$, then we may construct a set $\chi \in \mathcal{S}$ of size
$$
|\chi| \geq \exp\left( \frac{r d_2}{16\gamma^2} \right)
$$
with the following properties: (1) for all $X \in \chi$, each entry has $|X_{ij}| = \alpha \gamma$; and (2) for all $X^{(i)}$,$X^{(j)} \in \chi$, $i\neq j$,
$
\|X^{(i)} - X^{(j)} \|_F^2 > \alpha^2 \gamma^2 d_1 d_2/2.
$
\label{packingset}
\end{lemma}

\begin{lemma}
Given $K$ categories, the KL divergence for two categorical probability distributions with parameter $(p_1,\ldots,p_K)$ and $(q_1,\ldots,q_K)$,  is upper-bounded by
\begin{equation}
\begin{split}
&D((x_1,\ldots,x_K)\|(y_1,\ldots,y_K)) \leq \\
& \sum_{k=1}^{K-1} [(x_k-y_k)^2+(x_k y_k -x_k^2)(1-y_K)\\
&~~~~~~+(x_k y_k-y_k^2)(1-x_K)]/[y_k(1-\sum_{i=1}^{K-1}y_i)].
\end{split} \nonumber
\end{equation}
\label{KLdivergence}
\end{lemma}
%
The following theorem shows the existence of a worst case scenario in which no algorithm can reconstruct the matrix with an arbitrarily small error:
\begin{theorem}[Lower bound]
Fix $\alpha$, $r$, $d_1$, and $d_2$ to be such that $\alpha, d_1, d_2 \geq 1$, $r \geq 4$ and $\alpha^2 r \max\{d_1,d_2\} \geq C_0$. Assume that each $f_k$ is decreasing, and its derivative $f_k'$ is increasing in $[\alpha-1/4,\alpha]$ for all $k\in [K-1]$, and $f_K(\alpha-1/4)>1/2$. Let $\Omega$ be any subset of $[d_1] \times [d_2]$ with cardinality $m$. Let $Y$ be as in (\ref{observationmodel}) and $\beta_{\alpha}^+$ be as in (\ref{constant2}). Consider any algorithm which, for any $M \in \mathcal{S}$, returns an estimator $\widehat{M}$. Then there exists $M \in \mathcal{S}$ such that with probability at least $3/4$,
\begin{equation}
\begin{split}
&\frac{1}{d_1 d_2} R(M, \widehat{M}) \\
&\quad \geq \min\left\{C_1, C_2 \frac{\alpha}{\sqrt{K\beta_{\alpha}^+}} \sqrt{\frac{r\max\{d_1,d_2\}}{m}}\right\},
\end{split}
\label{lowerbound}
\end{equation}
as long as the right-hand side of (\ref{lowerbound}) exceeds $r\alpha^2 /\min\{d_1,d_2\}$, where $C_0, C_1,C_2$ are absolute constants.
\label{maintheorem2}
\end{theorem}


\begin{remark}
The ratio between the upper bound and lower bound is proportion to $K^{3/2}L_{\alpha}\sqrt{\beta_{\alpha}^+}/\beta_{\alpha}^-$. However, if we carefully construct $f_k$ as in Remark 2 so that $K L_{\alpha}/(\beta_{\alpha}^-)$ is less than some absolute constant, the gap between the upper and the lower bound can be reduced to a factor that is on the order of $O(\sqrt{K})$.
\end{remark}

\section{Numerical examples}

To test the performance of the regularized maximum estimator on real data, we consider the MovieLens dataset (can be downloaded at http://www.grouplens.org). The dataset contains $10^5$ movie ratings from $942$ users can be viewed as a $942$-by-$1683$ matrix whose entries take value $\{1,2,3,4,5\}$. We first randomly select $5000$ ratings to fit a multi-nomial logit regression model, and then solve the optimization problem (\ref{optimization_problem}) using another randomly selected $95,000$ ratings as observed entries.
Finally, we use the remaining $5000$ ratings as test data and compute the average difference between the true rating and the predicted ratings based on the recovered matrix, as shown in the first row of Table \ref{tab1}. We compare the result with that obtained from the conventional matrix completion by rounding the recovered entries to $\{1, 2, 3, 4, 5\}$. The results are shown in the second row of Table \ref{tab1}.

%

The results show that our method performs better when the original rating is larger than $3$, and has better overall performance. This can possibly be explained by the following reasoning. Our multinomial logistic model is fitted using the training data, and in the MovieLens dataset there are relatively few ratings with values 1, 2 and 3. Therefore, the fitted link functions $f_k$ for $k = 1, 2, 3$, have much less accuracy than $k = 4$ and $5$. Hence, we can see that the categorical matrix completion performs better when the true rating is 4 and 5.

%

\begin{table}[h]
{\scriptsize
\caption{Average differences between true ratings $M_{ij}$ and recovered ratings $\hat{M}_{ij}$, when they are both in $\{1,2,3,4, 5\}$.
}
\begin{center}
  \begin{tabular}{ | c | c | c | c | c | c || c |}
    \hline
    Original Rating & 1 & 2 & 3 & 4 & 5 & Overall \\ \hline
    Categorical & 1.537 & 0.958 & 0.461 & 0.489 & 0.986 & 0.708 \\ \hline
    Real-valued (conventional) & 0.039 & 0.119 & 0.428 &1.159&1.244 &0.783\\
    \hline
  \end{tabular}
  \label{tab1}
\end{center}}
\end{table}

\section{Discussions}

We have studied a nuclear norm regularized maximum likelihood estimator for categorical matrix completion, as well as presented an upper bound and an information theoretic lower bound for our proposed estimator. Our upper and lower bounds meet up to a constant factor $\mathcal{O}(K^{3/2})$ where $K$ is the fixed number of categories, and this factor can become $\mathcal{O}(K^{1/2})$ in some special cases. 
Our current formulation assumes that the input variables form a low-rank matrix and each response is linked only to one corresponding input variable. Future extension may include a formulation that allows more general link functions with multiple input variables and exploits a low-rank tensor structure.

\section*{Acknowledgement}
The authors would like to thank Prof. Mark Davenport for stimulating discussions. This work is partially supported by NSF grant CCF-1442635.

\bibliography{K_bit_MatrixCompletion}

\clearpage

\appendices

\section*{Proofs}

\begin{proof}[Proof of Lemma \ref{firstlemma}]

In order to prove the lemma, we let
$\epsilon_{ij}$ are i.i.d. Rademacher random variables.
    In the following derivation, the first inequality uses the Radamacher symmetrization argument (Lemma 6.3 in \cite{ledoux1991probability}) and the second inequality is due to the power mean inequality: $(\sum_{i=1}^K c_i)^h \leq K^{h-1}(\sum_{i=1}^K c_i^h)$ if $c_i \geq 0, \forall 1\leq i\leq K$ and $h\geq 1$. Then we have
        \begin{equation}
        \begin{split}
        &\mathbb{E}\left[\sup_{X\in \mathcal{S}} \left| \bar{F}_{\Omega, Y}(X) - \mathbb{E}\bar{F}_{\Omega, Y}(X) \right|^h \right] \\
        \leq& 2^h \mathbb{E}\left[ \sup_{X\in \mathcal{S}} \left| \sum_{i,j} \epsilon_{ij} \mathbb{I}_{[(i,j)\in \Omega]} \left(\sum_{k=1}^K \mathbb{I}_{[Y_{ij} = a_k]} \log \frac{f_k(X_{ij})}{f_k(0)} \right) \right|^h \right]\\
        =& 2^h \mathbb{E}\left[ \sup_{X\in \mathcal{S}} \left| \sum_{k=1}^K \sum_{i,j} \epsilon_{ij} \mathbb{I}_{[(i,j)\in \Omega]} \mathbb{I}_{[Y_{ij} = a_k]} \log \frac{f_k(X_{ij})}{f_k(0)} \right|^h \right]\\
        \leq& 2^h K^{h-1} \sum_{k=1}^K \mathbb{E}\left[ \sup_{X\in \mathcal{S}} \left| \sum_{i,j} \epsilon_{ij} \mathbb{I}_{[(i,j)\in \Omega]} \mathbb{I}_{[Y_{ij} = a_k]} \log \frac{f_k(X_{ij})}{f_k(0)} \right|^h \right]\\
        \leq & 2^h K^{h-1} \sum_{k=1}^K \left( \mathbb{E} \left[ \max_{i,j} \left|  \mathbb{I}_{[Y_{ij} = a_k]} \right|^h \right] \cdot \right. \\
        & \left. \mathbb{E}\left[ \sup_{X\in \mathcal{S}} \left| \sum_{i,j} \epsilon_{ij} \mathbb{I}_{[(i,j)\in \Omega]} \log \frac{f_k(X_{ij})}{f_k(0)} \right|^h \right] \right)\\
        \leq &2^h K^{h-1} \sum_{k=1}^K  \mathbb{E}\left[ \sup_{X\in \mathcal{S}} \left| \sum_{i,j} \epsilon_{ij} \mathbb{I}_{[(i,j)\in \Omega]} \log \frac{f_k(X_{ij})}{f_k(0)} \right|^h \right]
        \label{wholepart}
        \end{split}
        \end{equation}
    where the expectation are over both $\Omega$ and $Y$.

    In the following, we will use contraction principle to further bound the first term of (\ref{wholepart}). By the definition of $L_\alpha^{(k)}$, we know that
    $$
    \frac{1}{L_\alpha^{(k)}} \log\frac{f_k(x)}{f_k(0)}
    $$
    are contractions that vanish at $0$ for all $k =1,2,\ldots, K$.
    By Theorem 4.12 in \cite{ledoux1991probability} and using the fact that $|\langle A,B \rangle| \leq \|A\|\|B\|_*$, we have
           \begin{equation}
        \begin{split}
        &\mathbb{E}\left[\sup_{X\in \mathcal{S}} \left| \bar{F}_{\Omega, Y}(X) - \mathbb{E}\bar{F}_{\Omega, Y}(X) \right|^h \right] \\
        &\leq 2^h K^{h-1} \sum_{k=1}^K \left(2L_\alpha^{(k)}\right)^h \mathbb{E}\left[ \sup_{X\in \mathcal{S}} \left| \sum_{i,j} \epsilon_{ij} \mathbb{I}_{[(i,j)\in \Omega]} X_{ij} \right|^h \right] \\
          &\leq (4K)^h \left(\max_{1\leq k\leq K} L_\alpha^{(k)} \right)^h \mathbb{E}\left[ \sup_{X \in \mathcal{S}} \|E \circ \Delta_{\Omega} \|^h \|X\|_*^h \right] 
           \end{split}\nonumber
        \end{equation}
         \begin{equation}
        \begin{split}
        &\leq (4K)^h \left(L_{\alpha} \right)^h \left(\alpha \sqrt{r d_1 d_2}\right)^h \mathbb{E}\left[ \|E \circ \Delta_{\Omega} \|^h \right],
        \end{split}\nonumber
        \end{equation}
    where $E$ denotes the matrix with entries given by $\epsilon_{ij}$, $\Delta_{\Omega}$ denotes the indicator matrix for $\Omega$ and $\circ$ denotes the Hadamard product.

    To bound $\mathbb{E} \left[\|E \circ \Delta_{\Omega} \|^h\right]$, we can use the result from \cite{davenport20141} if we take $h=\log(d_1 d_2) \geq 1$:
    \begin{equation}
    \begin{split}
    &\mathbb{E} \left[\|E \circ \Delta_{\Omega} \|^h\right] \\
    \leq& C_0 \left(2(1+\sqrt{6})\right)^h \left( \sqrt{\frac{m(d_1+d_2)+d_1 d_2 \log(d_1 d_2)}{d_1 d_2}} \right)^h
    \end{split}\nonumber
    \end{equation}
    for some constant $C_0$.

    Moreover when $C'\geq 8\left(1+\sqrt{6}\right)e$,
    $$
    C_0\left( \frac{8(1+\sqrt{6})}{C'} \right)^{\log(d_1 d_2)} \leq \frac{C_0}{d_1 d_2}
    $$

    Therefore we can use Markov inequality to see that
    \begin{equation}
    \begin{aligned}
    &\mathbb{P} \left\{ \sup_{X\in \mathcal{S}} \left| F_{\Omega, Y}(X)-\mathbb{E}F_{\Omega, Y}(X)\right| \right. \\
    & \quad \left. \geq C'K\left( \alpha\sqrt{r} \right)  L_{\alpha} \cdot \right. \\
    & \quad \left. \left(\sqrt{m(d_1+d_2)+d_1 d_2 \log(d_1 d_2)}\right) \right\} \\
    =&\mathbb{P} \left\{ \sup_{X\in \mathcal{S}} \left| F_{\Omega, Y}(X)-\mathbb{E}F_{\Omega, Y}(X)\right|^h \right. \\
    & \quad \left. \geq \left(C'K\left( \alpha\sqrt{r} \right) L_{\alpha} \cdot \right.\right. \\
    & \quad \left. \left. \left(\sqrt{m(d_1+d_2)+d_1 d_2 \log(d_1 d_2)}\right)\right)^h \right\} \\
    \leq & \mathbb{E}\left[\sup_{X\in \mathcal{S}} \left| F_{\Omega, Y}(X) - EF_{\Omega, Y}(X) \right|^h \right]/\\
    &
    \{\left(C'K\left( \alpha\sqrt{r} \right) L_{\alpha} \right.\cdot \\
    &\left.\left(\sqrt{m(d_1+d_2)+d_1 d_2 \log(d_1 d_2)}\right) \right)^h \}
    \leq  \frac{C}{d_1  d_2}, \nonumber
    \end{aligned}
    \end{equation}
    where $C' \geq 8(1+\sqrt{6})e$ and $C$ are absolute constants.

\end{proof}

\begin{proof}[Proof of Lemma \ref{secondlemma}]
Assuming $x$ is any entry in $M$ and $y$ is any entry in $\widehat{M}$, then $-\alpha \leq x,y \leq \alpha$ and by
the mean value theorem there exists $\xi_k \in [x,y]$ for each $k \in [K]$ such that
$$
\sqrt{f_k(x)} - \sqrt{f_k(y)} = \frac{f_k'(\xi_k)}{2\sqrt{f_k(\xi_k)}}(x-y).
$$
By the assumption of $f$, there exist at least one $k\in [K]$ such that $f_k'(\xi_k) \neq 0$. 
Then
\begin{equation}
\begin{split}
&d_H^2\left( \left(f_1(x),\ldots,f_K(x)\right) \| \left( f_1(y),\ldots, f_K(y) \right) \right)\\
=& \sum_{k=1}^K \left( \sqrt{f_k(x)} - \sqrt{f_k(y)} \right)^2 
= \frac{(x-y)^2 }{4} \sum_{k=1}^K \frac{(f_k'(\xi_k))^2}{f_k(\xi_k)} \nonumber
\end{split}
\end{equation}
\begin{equation}
\begin{split}
\geq & \frac{(x-y)^2 }{4}  \left(\max_{1\leq k \leq K} \frac{(f_k'(\xi_k))^2}{f_k(\xi_k)}\right) \\
\geq& \frac{(x-y)^2 }{4} \left(\inf_{|\xi|\leq \alpha} \left(\max_{1\leq k \leq K} \frac{(f_k'(\xi))^2}{f_k(\xi)}\right)\right)
\end{split}
\label{se1}
\end{equation}
Then the lemma is proven by summing across all entries and dividing by $d_1 d_2$.

\end{proof}

\begin{proof}[Proof of Lemma \ref{KLdivergence}]
Let $z_k = y_k - x_k$ for each $k \in [K]$, then
\begin{equation}
\begin{split}
&D\left( (x_1,\ldots,x_K) \| (y_1,\ldots,y_K) \right) \\
=& D\left( (x_1,\ldots,x_K) \| (x_1+z_1,\ldots,x_K+z_K) \right) \\
=& \sum_{k=1}^K x_k \log\frac{x_k}{x_k+z_k}.
\end{split}\nonumber
\end{equation}
And then we have for each $k \in [K]$
$$
\frac{\partial  D\left((x_1,\ldots,x_K) \| (x_1+z_1,\ldots,x_K+z_K) \right)}{\partial z_k} = -\frac{x_k}{x_k+z_k}.
$$

By mean value theorem, we have
\begin{equation}
\begin{split}
&D\left( (x_1,\ldots,x_K) \| (x_1+z_1,\ldots,x_K+z_K) \right) \\
=& -\sum_{k=1}^K \frac{x_k z_k}{x_k+c z_k}
\end{split}
\label{meanvalue1}
\end{equation}
for some $c \in [0,1]$.
Since for each $k \in [K]$
$$
\left( -\frac{x_k z_k}{x_k+ c z_k} \right)' = \frac{x_k z_k^2}{(x_k+c z_k)^2} \geq 0,
$$
the right-hand side of (\ref{meanvalue1}) is an increasing function in $c$ and hence
$$
D\left( (x_1,\ldots,x_K) \| (y_1,\ldots,y_K) \right) \leq \sum_{k=1}^K \frac{x_k}{y_k}(x_k-y_k).
$$
Noting that $x_K=1-\sum_{i=1}^{K-1}x_i$ and $y_K=1-\sum_{i=1}^{K-1}y_i$, we have
\begin{equation}
\begin{split}
&\sum_{k=1}^K \frac{x_k}{y_k}(x_k-y_k) 
= \sum_{k=1}^{K-1}\left( \frac{x_k}{y_k} - \frac{1-\sum_{i=1}^{K-1}x_i}{1-\sum_{i=1}^{K-1}y_i}\right) (x_k-y_k)\\
=& \sum_{k=1}^{K-1}[(x_k-y_k)^2+x_k y_k (2-x_K-y_K) \\
&\qquad -x_k^2(1-y_K)-y_k^2(1-x_K)]/[y_k\left(1-\sum_{i=1}^{K-1}y_i\right)] \nonumber
\end{split}
\end{equation}
\begin{equation}
\begin{split}
=& \sum_{k=1}^{K-1} [(x_k-y_k)^2+(x_k y_k -x_k^2)(1-y_K)\\
&\qquad +(x_k y_k-y_k^2)(1-x_K)]/[y_k(1-\sum_{i=1}^{K-1}y_i)],
\end{split}
\end{equation}
where the last inequality uses the fact that $0\leq x_k,y_k \leq 1 ~ \forall k\in [K]$.

\end{proof}

\begin{proof}[Proof of Theorem \ref{maintheorem}]
First, note that
    \begin{equation}
    \bar{F}_{\Omega,Y}(X) - \bar{F}_{\Omega,Y}(M) = \sum_{(i,j)\in \Omega} \sum_{k=1}^K \mathbb{I}_{(Y_{ij}=a_k)} \log\frac{f_k(X_{ij})}{f_k(M_{ij})}.  \nonumber
    \end{equation}
    Then for any $X \in \mathcal{S}$,
    \begin{equation}
    \begin{aligned}
    &\mathbb{E}\left[\bar{F}_{\Omega,Y}(X) - \bar{F}_{\Omega,Y}(M)\right] 
    = \frac{m}{d_1 d_2} \sum_{i,j}\sum_{k=1}^K f_k(M_{ij})\log\frac{f_k(X_{ij})}{f_k(M_{ij})} \\
    &=-mD(f(M)\|f(X)).
    \end{aligned}
    \end{equation}

    For $M\in \mathcal{S}$,  we know $\widehat{M} \in \mathcal{S}$ and $F_{\Omega,Y}(\widehat{M}) \geq F_{\Omega,Y}(M) $. Thus we  write
    \begin{equation*}
    \begin{aligned}
    0&\leq F_{\Omega,Y}(\widehat{M}) - F_{\Omega,Y}(M) = \bar{F}_{\Omega,Y}(\widehat{M}) - \bar{F}_{\Omega,Y}(M) \\
    &= \bar{F}_{\Omega,Y}(\widehat{M}) + \mathbb{E}\bar{F}_{\Omega,Y}(\widehat{M}) - \mathbb{E}\bar{F}_{\Omega,Y}(\widehat{M}) \\
    & ~~ + \mathbb{E}\bar{F}_{\Omega,Y}(M) - \mathbb{E}\bar{F}_{\Omega,Y}(M) - \bar{F}_{\Omega,Y}(M) \\
    &\leq \mathbb{E}\left[ \bar{F}_{\Omega,Y}(\widehat{M}) - \bar{F}_{\Omega,Y}(M) \right] + \\
    & ~~ \left| \bar{F}_{\Omega,Y}(\widehat{M})-\mathbb{E}\bar{F}_{\Omega,Y}(\widehat{M})\right| + \left|\bar{F}_{\Omega,Y}(M) - \mathbb{E}\bar{F}_{\Omega,Y}(M)\right| \\
    &\leq -m D(f(M)\|f(\widehat{M})) + 2\sup_{X\in \mathcal{S}} \left| \bar{F}_{\Omega,Y}(X) - \mathbb{E}\bar{F}_{\Omega,Y}(X)\right|.
    \end{aligned}
    \end{equation*}
    Applying Lemma \ref{firstlemma}, we obtain that with probability at least $\left(1-C /(d_1 d_2)\right)$,
    \begin{equation*}
    \begin{split}
    0 &\leq -m D(f(M)\|f(\widehat{M})) \\
    &+ 2C'K\left( \alpha\sqrt{r} \right)  L_{\alpha} \cdot 
    \left(\sqrt{m(d_1+d_2)+d_1 d_2 \log(d_1 d_2)}\right).
    \end{split}
    \end{equation*}
    After rearranging terms and applying the fact that $\sqrt{d_1 d_2} \leq d_1 + d_2$, we obtain
    \begin{equation}
    \begin{split}
   &D(f(M)\|f(\widehat{M}))\leq \\
    \hspace{-0.3in} & 2C'K\left( \alpha\sqrt{r} \right)  L_{\alpha} \cdot  \left(\sqrt{\frac{d_1+d_2}{m}} \sqrt{1 +\frac{(d_1+ d_2) \log(d_1 d_2)}{m}}\right).
    \label{theoremuse1}
    \end{split}
    \end{equation}
    Note that the KL divergence can be bounded below by the Hellinger distance (Chapter 3 in \cite{pollard2002user}):
    $
    d_H^2(x,y) \leq D(x\|y).
    $
    Thus from (\ref{theoremuse1}), we obtain
    \begin{equation}
    \begin{split}
    &d_H^2(f(M),f(\widehat{M})) \leq  2C'K\left( \alpha\sqrt{r} \right)  L_{\alpha} \cdot \\
    &~~~ \left(\sqrt{\frac{d_1+d_2}{m}} \sqrt{1 +\frac{(d_1+ d_2) \log(d_1 d_2)}{m}}\right).
    \end{split}
    \end{equation}
    Finally, Theorem \ref{maintheorem} is proved by applying Lemma \ref{secondlemma}.
\end{proof}

\begin{proof}[Proof of Theorem \ref{maintheorem2}]

We will prove by contradiction.
Lemma \ref{packingset} and Lemma \ref{KLdivergence} are used in the proof. Without loss of generality,  assume $d_2 \geq d_1$. Choose $\epsilon > 0$ such that
$$
\epsilon^2 = \min\left\{ \frac{1}{1024}, C_2 \frac{\alpha }{\sqrt{K\beta_{\alpha}^+}} \sqrt{\frac{r d_2}{m}}\right\},
$$
where $C_2$ is an absolute constant that will be be specified later.
First, choose $\gamma$ such that $\frac{r}{\gamma^2}$ is an integer and
$$
\frac{4\sqrt{2}\epsilon}{\alpha} \leq \gamma \leq \frac{8\epsilon}{\alpha} \leq \frac{1}{4\alpha}
$$
We may make such a choice because
$$
\frac{\alpha^2 r}{64\epsilon^2} \leq \frac{r}{\gamma^2} \leq \frac{\alpha^2 r}{32\epsilon^2}
$$
and
$$
\frac{\alpha^2 r}{32\epsilon^2} - \frac{\alpha^2 r}{64\epsilon^2}  = \frac{\alpha^2 r}{64\epsilon^2} > 4\alpha^2 r > 1.
$$
Furthermore, since we have assumed that $\epsilon^2$ is larger than $C r\alpha^2/d_1$, $r/\gamma^2 \leq d_1$ for an appropriate choice of $C$.
Let $\chi'_{\alpha/2, \gamma}$ be the set defined in Lemma \ref{packingset}, by replacing $\alpha$ with $\alpha/2$ and with this choice of $\gamma$. Then we can construct a packing set
$\chi$ of the same size as $\chi'_{\alpha/2, \gamma}$ by defining
$$
\chi \triangleq \left\{ X' + \alpha\left(1-\frac{\gamma}{2}\right) \textbf{1}_{d_1 \times d_2} : X' \in \chi'_{\alpha/2, \gamma} \right\}.
$$
The distance between pairs of elements in $\chi$ is bounded since
\begin{equation}
\|X^{(i)} - X^{(j)} \|_F^2 \geq \frac{\alpha^2}{4}\frac{\gamma^2 d_1 d_2}{2} \geq 4 d_1 d_2 \epsilon^2.
\end{equation}
Define $\alpha' \triangleq (1-\gamma)\alpha$, then every entry of $X \in \chi$ has $X_{ij} \in \{\alpha, \alpha'\}$. Since we have assumed $r \geq 4$, for every $X \in \chi$, we have
\begin{align*}
\|X\|_* &= \| X' + \alpha\left(1-\frac{\gamma}{2}\right) \textbf{1}_{d_1 \times d_2}\|_* \\
&\leq \|X'\|_* + \alpha(1-\frac{\gamma}{2})\sqrt{d_1 d_2}\\
&\leq \frac{\alpha}{2} \sqrt{r d_1 d_2} + \alpha \sqrt{d_1 d_2}
\leq \alpha \sqrt{r d_1 d_2},
\end{align*}
for some $X' \in \chi'_{\alpha/2, \gamma}$.
Since the $\gamma$ we choose is less than $1/2$, $\alpha'$ is greater than $\alpha/2$. Therefore, from the assumption that $\beta \leq \alpha/2$, we conclude that $\chi \subset \mathcal{S}$.

Now consider an algorithm that for any $X \in \mathcal{S}$ returns $\widehat{X}$ such that
\begin{equation}
\frac{1}{d_1 d_2} \|X-\widehat{X}\|_F^2 < \epsilon^2
\label{assumption}
\end{equation}
with probability at least $1/4$. Next, we will show this leas to an contradiction. Let
$$
X^* = \arg \min_{X^{(i)} \in \chi} \|X^{(i)} - \widehat{X}\|_F^2,
$$
by the same argument as that in \cite{davenport20141}, we have $X^* = X$ as long as (\ref{assumption}) holds. Using the assumption that (\ref{assumption}) holds with probability at least
$1/4$, we have
\begin{equation}
\mathbb{P} (X^* \neq X) \leq \frac{3}{4}.
\label{inequality1}
\end{equation}
Using a generalized Fano's inequality for the KL divergence in \cite{yu1997assouad}, we have
\begin{equation}
\mathbb{P} (X^* \neq X) \geq 1- \frac{\max_{X^{(k)} \neq X^{(l)}} D(Y_{\Omega}|X^{(k)} \| Y_{\Omega}|X^{(l)})+1}{\log |\chi|}.
\label{inequality2}
\end{equation}
Define
\[
D \triangleq D(Y_{\Omega}|X^{(k)}\|Y_{\Omega}|X^{(l)}) = \sum_{(i,j) \in \Omega} D(Y_{ij}|X^{(k)}_{ij}\| Y_{ij}|X^{(l)}_{ij} ).
\]
Because the entries of $X^{(k)}$ and $X^{(l)}$ are $\alpha$ or $\alpha'$, from Lemma \ref{KLdivergence}, we have
\begin{equation}
\begin{split}
D \leq& m\sum_{k=1}^{K-1}\frac{\left(f_k(\alpha)-f_k(\alpha')\right)^2}{\min \left\{f_k(\alpha)f_K(\alpha), f_k(\alpha')f_K(\alpha') \right\}} \\
+& m\sum_{k=1}^{K-1}\frac{(f_k(\alpha)f_k(\alpha')-f_k^2(\alpha))(1-f_K(\alpha'))}{\min \left\{f_k(\alpha)f_K(\alpha), f_k(\alpha')f_K(\alpha') \right\}} \\
+& m\sum_{k=1}^{K-1}\frac{(f_k(\alpha)f_k(\alpha')-f_k^2(\alpha'))(1-f_K(\alpha))}{\min \left\{f_k(\alpha)f_K(\alpha), f_k(\alpha')f_K(\alpha') \right\}}.
\end{split}
\end{equation}
Considering the assumptions of $f_k$ in the theorem for all $k\in [K]$, we can further bound $D$ as follows:
\begin{equation}
\begin{split}
D \leq& m\sum_{k=1}^{K-1}\frac{(f_k'(\alpha))^2(\alpha-\alpha')^2}{f_k(\alpha)f_K(\alpha')} \\
+&  m\sum_{k=1}^{K-1} \frac{(f_k(\alpha')-f_k(\alpha))(1-f_K(\alpha'))}{f_K(\alpha')} \\
\leq& m(K-1) \beta_{\alpha}^+ \frac{(\gamma \alpha)^2}{f_K(\alpha')} + m\frac{(1-f_K(\alpha'))^2}{f_K(\alpha')} \\
\leq& 2mK \beta_{\alpha}^+ (\gamma \alpha)^2 + \frac{m}{2},
\end{split}
\end{equation}
where we use the mean value theorem and facts that $f_k'(\alpha)\geq f_k'(\alpha')$, $f_k(\alpha)\leq f_k(\alpha')$ and $f_K(\alpha)\geq f_K(\alpha')$ in the first inequality and fact that $f_K(\alpha')>1/2$ in the third inequality.

Combining (\ref{inequality1}) and (\ref{inequality2}), we have that
\begin{equation}
\begin{split}
\frac{1}{4} &\leq 1-\mathbb{P}(X \neq X^*) \leq \frac{D+1}{\log |\chi|} \\
&\leq 16\gamma^2 \left(\frac{D+1}{rd_2} \right) \leq 1024\epsilon^2 \left(\frac{128mK \beta_{\alpha}^+ \epsilon^2 + \frac{m}{2}+1}{\alpha^2 rd_2}\right).
\end{split}
\label{contradiction}
\end{equation}
Suppose $128mK \beta_{\alpha}^+ \epsilon^2 + \frac{m}{2} \leq 1$, then with (\ref{contradiction}), we have
$$
\frac{1}{4} \leq 1024 \epsilon^2 \frac{2}{\alpha^2 r d_2},
$$
which implies that $\alpha^2 r d_2 \leq 32$. Then if we set $C_0>32$, this leads to a contradiction.
Next, suppose $128mK \beta_{\alpha}^+ \epsilon^2 + \frac{m}{2} > 1$, then with (\ref{contradiction}), we have
$$
\frac{1}{4} < 1024\epsilon^2 \left(\frac{256mK \beta_{\alpha}^+ \epsilon^2 + m}{\alpha^2 rd_2}\right),
$$
thus,
$$
\epsilon^2 > \frac{-1+\sqrt{1+\frac{\alpha^2rd_2 K \beta_{\alpha}^+}{4m}}}{512K \beta_{\alpha}^+}.
$$
By using the fact that $\sqrt{a^2+b^2} \geq (a+b)/\sqrt{2}$ for any $a,b>0$, we have
$$
\epsilon^2 > \frac{\alpha}{1024\sqrt{2}\sqrt{K\beta_{\alpha}^+}}\sqrt{\frac{rd_2}{m}}
$$
Setting $C_2 \leq 1/1024\sqrt{2}$, this leads to a contradiction. Therefore, (\ref{assumption}) must be incorrect with probability at least $3/4$. This concludes our proof.

\end{proof}

\end{document}